\documentclass[12pt,oneside]{amsart}
\usepackage{amssymb,amsmath,latexsym,epsfig}

\usepackage{euler, amsfonts, amssymb, latexsym, epic}

\numberwithin{equation}{section}
\begin{document}


\newcommand{\az}{\alpha_{Z}}
\newcommand{\bz}{\beta_{Z}}
\newcommand{\ax}{\alpha_{\X}}\newcommand{\bx}{\beta_{\X}}
\newcommand{\popo}{\mathbb{P}^1 \times \mathbb{P}^1}
\newcommand{\pnpm}{\mathbb{P}^n \times \mathbb{P}^m}
\newcommand{\pnr}{\mathbb{P}^{n_1}\times \cdots \times \mathbb{P}^{n_r}}
\newcommand{\pnk}{\mathbb{P}^{n_1}\times \cdots \times \mathbb{P}^{n_k}}
\newcommand{\prthree}{\mathbb{P}^1 \times \mathbb{P}^1 \times \mathbb{P}^1}
\newcommand{\Iz}{I_{Z}}
\newcommand{\Ix}{I_{\X}}
\newcommand{\C}{\mathcal{C}}
\newcommand{\Y}{\mathbb{Y}}
\newcommand{\Z}{\mathbb{Z}}
\newcommand{\Zr}{\mathbb{Z}_{red}}
\newcommand{\N}{\mathbb{N}}
\newcommand{\pr}{\mathbb{P}}
\newcommand{\X}{\mathbb{X}}
\newcommand{\F}{\mathbb{F}}

\newcommand{\supp}{\operatorname{Supp}}
\newcommand{\depth}{\operatorname{depth}}
\newcommand{\ol}{\overline{L}}
\newcommand{\Ssx}{\mathcal S_{\X}}
\newcommand{\Ss}{\mathcal S}
\newcommand{\dtz}{\Delta H_{Z}}
\newcommand{\dtc}{\Delta^{C} H_{Z}}
\newcommand{\dtcc}{\Delta^{C} H_{Z_{ij}}}
\newcommand{\dt}{\Delta}
\newcommand{\B}{\mathcal{B}}
\newcommand{\ay}{\alpha_{Y}}
\newcommand{\by}{\beta_{Y}}
\newcommand{\ui}{\underline{i}}
\newcommand{\ua}{\underline{\alpha}}
\newcommand{\uj}{\underline{j}}

\newtheorem{theorem}{Theorem}[section]
\newtheorem{corollary}[theorem]{Corollary}
\newtheorem{proposition}[theorem]{Proposition}
\newtheorem{lemma}[theorem]{Lemma}
\newtheorem{alg}{Algorithm}
\newtheorem{question}{Question}
\newtheorem{problem}{Problem}
\newtheorem{conjecture}[theorem]{Conjecture}

\theoremstyle{definition}
\newtheorem{definition}[theorem]{Definition}
\newtheorem{remark}[theorem]{Remark}
\newtheorem{example}[theorem]{Example}
\newtheorem{acknowledgement}{Acknowledgment}
\newtheorem{notation}[theorem]{Notation}
\newtheorem{construction}[theorem]{Construction}


\title{ACM sets of points in multiprojective space}
\thanks{Version: Revised Final, Nov. 28, 2008}
\author{Elena Guardo}
\address{Dipartimento di Matematica e Informatica\\
Viale A. Doria, 6 - 95100 - Catania, Italy}
\email{guardo@dmi.unict.it}

\author{Adam Van Tuyl}
\address{Department of Mathematics \\
Lakehead University \\
Thunder Bay, ON P7B 5E1, Canada}
\email{avantuyl@sleet.lakeheadu.ca}


\begin{abstract}
If $\X$ is a finite set of points in a multiprojective space
$\pnr$ with $r\geq 2$, then $\X$ may or may not be arithmetically Cohen-Macaulay
(ACM).  For sets of points in $\popo$ there are several classifications of
the ACM sets of points. In this paper we investigate the natural
generalizations of these classifications to an arbitrary
multiprojective space.  We show that each classification for
ACM points in $\popo$ fails to extend to the general case.
We also give some new necessary and sufficient conditions
for a set of points to be ACM.
\end{abstract}
\maketitle


\section{Introduction}

Let $\X$ be a finite set of points in a multiprojective space
$\pnr$, and let $R/I_{\X}$ denote the associated $\N^r$-graded
coordinate ring. When $r = 1$, then $R/I_{\X}$ is always
Cohen-Macaulay.  On the other hand, if $r \geq 2$, then $R/I_{\X}$
may or may not be Cohen-Macaulay.  More precisely, we know that
$\dim R/I_{\X} = r$, the number of projective spaces. However, the
depth of $R/I_{\X}$ may take on any value in the set
$\{1,\ldots,r\}$. When $R/I_{\X}$ is Cohen-Macaulay, that is,
$\depth R/I_{\X} = \dim  R/I_{\X} = r$, then $\X$ is called an
{\bf arithmetically Cohen-Macaulay} (ACM) set of points.

Because a set of points in a multiprojective space
may or may not be ACM, a natural problem arises:

\begin{problem} Find a classification of ACM sets of points in $\pnr$
for $r \geq 2$.
\end{problem}

Little is known about this problem except in the case that $\X \subseteq
\popo$. In this situation there are several classifications.
Giuffrida, Maggioni, and Ragusa \cite{GuMaRa}, who helped to
initiate the study of points in multiprojective spaces
(see, for example \cite{GuMaRa2,GuMaRa3,Gu3,HVT,m1,SVT,VT1,VT2,VT3} for more on these points), provided
the first classification.  They showed that ACM sets of points in
$\popo$ can be classified via their Hilbert functions. The two
authors \cite{Gu2,VT2} independently
gave geometric classifications of ACM sets of points in $\popo$.
More recently, L. Marino \cite{m3} used the notion of a separator to provide
a new classification of ACM sets of points in $\popo$.

In this paper we will consider the natural generalizations of the
above classifications to an arbitrary multiprojective space.
As we shall show, these natural generalizations no longer
classify ACM sets of points, thus suggesting a solution
to Problem 1 is quite subtle.  We give a partial answer
to Problem 1 by giving some necessary and
sufficient conditions for a set of points to be ACM in $\pnr$.

Before proceeding, we should point out that Problem 1 is a refinement
of the following question:  if $X_1,\ldots,X_s$ are linear subspaces of $\pr^n$,
then when is $X = \bigcup_{i=1}^s X_i$ ACM?  To see this, note
that if we consider only
the graded structure of $R/\Ix$, then the defining ideal of each point
is also the defining ideal of a linear subspace in a projective space.
This paper, therefore, can be seen as one attack on this more general
question.  Alternatively, this paper can viewed as part
of the program to understand when a multigraded ring is Cohen-Macaulay
(for example, see \cite{CCH}).

We now expand upon the results of this paper.
We start in Section 2 by recalling the relevant results and definitions
about sets of points in a multiprojective space.  In Section 3 we study
the Hilbert function of an ACM set of points.  As mentioned above,
ACM sets of points in $\popo$ can be classified
via their Hilbert function $H_{\X}$; precisely, $\X$ is ACM in $\popo$ if and only
if $\Delta H_{\X}$, a generalized first difference function, is the
Hilbert function of a bigraded artinian quotient of $k[x_1,y_1]$.
One direction of this characterization extends to any multiprojective space,
as first proved by the second author \cite{VT2}:

\begin{theorem}[Theorem \ref{halfACMclassify}] \label{theorem1}
Let $\X \subseteq \pnr$ be a finite set of points
with Hilbert function $H_{\X}$.
If $\X$ is ACM, then $\Delta H_{\X}$ is the Hilbert function
of an $\N^r$-graded artinian quotient of $k[x_{1,1},\ldots,x_{1,n_1},\ldots,x_{r,1},\ldots,x_{r,n_r}]$
with $\deg x_{i,j} = e_i$.
\end{theorem}

\noindent It was not known whether the converse held  (in fact,
this question was raised in the second author's  thesis
\cite[Question 1.3.9]{VT}).  We give the first example (see
Example \ref{example1}) of a set of points in $\pr^2 \times \pr^2$
for which the converse fails and it can be extended to any
multiprojective space $\pnr$ with at least two $n_i$'s greater
than or equal to  $2$. This example together with Example \ref
{example2} demonstrates that we cannot expect a classification of
ACM sets of points based only upon the Hilbert function. However,
if all but one of the $n_i$'s equal one, we expect the converse of
Theorem \ref{halfACMclassify} to hold.  We give partial evidence
for this claim in Theorem \ref{depth=r-1} where we show that the
converse holds for sets of points $\X$ in $\pr^1 \times \cdots
\times \pr^1$ ($r$ times) with $\depth(R/I_\X) =r-1$.  In fact,
Theorem \ref{depth=r-1}, combined with Theorem \ref{theorem1},
will allow us to give a new proof of Giuffrida, Maggioni, and
Ragusa's result.

In Section 4 we examine how the geometry of a set of points influences
its ACMness.  If $\X$ is a set of points in $\pnpm$, we say that $\X$ satisfies {\bf property $(\star)$}
if whenever $P_1 \times Q_1$ and $P_2 \times Q_2$ are in $\X$ with $P_1 \neq P_2$
and $Q_1 \neq Q_2$, then either $P_1 \times Q_2$ or $P_2 \times Q_1$ (or both) are in $\X$.
The two authors independently showed (see \cite{Gu2,VT2}) that $\X$ is ACM in $\popo$
if and only if $\X$ satisfies property $(\star)$.  We extend one direction of this classification:

\begin{theorem}[Theorem \ref{geometriccondition}] \label{theorem2}
Let $\X \subseteq \pr^1 \times \pr^n$ be a finite set
of points.  If $\X$ satisfies property $(\star)$, then $\X$ is ACM.
\end{theorem}

\noindent
The converse, however, is false, as shown in Example \ref{example3} where we give an
example of an ACM set of points in $\pr^1 \times \pr^2$ which fails to
satisfy property $(\star)$.  At the end of Section 4, we show
how to use
Theorem \ref{theorem2} to easily
construct ACM sets of points in $\pr^1 \times \pr^n$.

In Section 5 we study the connection between the separators of a point
and the ACMness of a set of points.
If  $P \in \X$, then the multihomogeneous form $F\in R$ is
a {\bf separator for} $P$ if $F(P)\neq 0$ and $F(Q)=0$ for all $Q
\in \X \setminus\{P\}$.   The {\bf degree of a point} $P\in \X$ is the set
\begin{eqnarray*}
\deg_{\X}(P)&=&\min\{ \deg F ~|~ F ~~\text{is a separator for $P\in \X$}\}.
\end{eqnarray*}
(We are using the partial order on $\N^r$ defined by
$(i_1,\ldots,i_r) \succeq (j_1,\ldots,j_r)$ whenever $i_t \geq j_t$ for $t=1,\ldots,r$.)
Note that if $r\geq 1$, then we may have $|\deg_{\X}(P)| > 1$.
Separators for points in $\pr^1 \times \pr^1$ were first introduced by Marino
 \cite{m1},
who extended the original definition for points in $\pr^n$ due to Orecchia \cite{O}.
Marino has recently shown \cite{m3} that a set of points $\X \subseteq \popo$ is ACM
if and only if $|\deg_{\X}(P)| = 1$ for all $P \in \X$.  We show
that one direction of Marino's result holds in an arbitrary multiprojective space:

\begin{theorem}[Theorem \ref{uniqueMultsep}]\label{theorem3}
Let $\X \subseteq \pnr$ be a finite set of points.
If $\X$ is ACM, then $|\deg_{\X}(P)| = 1$ for every point $P \in \X$.
\end{theorem}

\noindent
The converse of Theorem \ref{theorem3} fails to hold;  Example \ref{example4}
gives an example of a set of points $\X \subseteq \pr^2 \times \pr^2$ where
every point $P \in \X$ has $|\deg_{\X}(P)| = 1$, but $\X$
fails to be ACM.

Finally, we note that examples of points in $\pnr$, especially the counterexamples
to the converses of Theorems \ref{theorem1}, \ref{theorem2}, and
\ref{theorem3},  play a
prominent role in this paper.
Instrumental in finding these examples was
the computer program {\tt
CoCoA} \cite{Co}.
To encourage further experimentation, our {\tt CoCoA} scripts and
examples can be found on the second author's
webpage.\footnote[1]{{\tt http://flash.lakeheadu.ca/$\sim$avantuyl/research/ACMexamples\_Guardo\_VanTuyl.html}}

\noindent {\bf Acknowledgments.} The second author acknowledges
the support provided by NSERC.  The authors also thank T\`ai
H\`a for his comments on an earlier draft of this project.


\section{Preliminaries}
We begin by recalling some relevant
results about points in a multiprojective space.
A more thorough introduction to points in a multiprojective space
can be found in \cite{VT1,VT2}. In this paper $k$ denotes
an algebraically closed field of characteristic zero.

We shall write $(i_1,\ldots,i_r)\in \N^r$ as $\ui$. We induce a
partial order on the set $\N^r$ by setting $(a_1,\ldots,a_r)
\succeq (b_1,\ldots,b_r)$ if $a_i \geq b_i$ for $i=1,\ldots,r$.
The coordinate ring of the {\bf multiprojective space} $\pnr$ is
the $\N^r$-graded ring
\[R=k[x_{1,0},\ldots,x_{1,n_{1}},x_{2,0},\ldots,x_{2,n_{2}},
\ldots,x_{r,0},\ldots,x_{r,n_{r}}]\] where $\deg x_{i,j} = e_i$,
the $i$th standard basis vector of $\N^r$. We use $R = k[x_0,\ldots,x_n,y_0,\ldots,y_m]$ if considering the
multiprojective space $\pnpm$.  If
\[ P =[a_{1,0}:\cdots:a_{1,n_1}] \times \cdots \times [a_{r,0}:\cdots:a_{r,n_r}] \in \pnr\]
is a {\bf point} in this space, then the ideal $I_P$ of $R$
associated to $P$ is a prime ideal of the form \[I_{P}
=(L_{1,1},\ldots,L_{1,n_1},\ldots,L_{r,1},\ldots,L_{r,n_r})\]
where $\deg L_{i,j}=e_i$ for $j=1,\ldots,n_i$. When
$\X=\{P_1,\ldots,P_s\}$ is a set of $s$ distinct points in $\pnr$,
then $I_{\X}=I_{P_1}\cap \cdots \cap I_{P_s}$, where $I_{P_i}$ is
the ideal associated to the point $P_i$, is the ideal generated by
all the multihomogeneous forms that vanish at all the points of $\X$.
The ideal $I_{\X}$ is a {\bf multihomogeneous} (or simply, {\bf
homogeneous}) ideal of $R$.

\begin{theorem} \label{depthdim}
Let $\X\subseteq \pnr$ be a finite set of points. Then
\[\dim R/I_{\X} = r  ~~\mbox{and} ~~ 1 \leq \depth R/I_{\X} \leq r.\]
In fact, for any $l \in \{1,\ldots,r\}$, there exists a set
of points $\X_l$ with $\depth R/I_{\X_l} = l$.
\end{theorem}
\begin{proof}
Because each prime ideal $I_{P_i}$ with $P_i \in \X$ has height
$\sum_{i=1}^r n_i$, it follows that $\dim R/I_{\X} =r$.
For the statement about the depth, see \cite[Proposition 2.6]{VT2}.
\end{proof}

\begin{definition} A set of points $\X \subseteq \pnr$ is {\bf arithmetically Cohen-Macaulay} (ACM) if
$\dim R/I_{\X} = \depth R/I_{\X} =r$, that is, if $R/{\Ix}$
is Cohen-Macaulay (CM).
\end{definition}

\begin{remark}
When $r = 1$, it is clear from Theorem \ref{depthdim} that $\X$ is
always ACM.   However, when $r \geq 2$, it is possible that
 $\depth R/I_{\X} < \dim R/I_{\X}$, and consequently,
$\X$ will not be ACM.  The sets $\X_l$ in Theorem
\ref{depthdim} can be constructed with the property
that $|\X_l| = 2$. \end{remark}

We will periodically require the following result which guarantees
the existence of a regular sequence of specific degrees.

\begin{theorem}\label{regularsequence}
Suppose that $\X$ is set of points in $\pnr$ with $\depth R/I_{\X}
= l \geq 1$. Then there exist elements
$\overline{L}_1,\ldots,\overline{L}_l$ in $R/\Ix$ such that
$L_1,\ldots,L_l$ give rise to a regular sequence in $R/\Ix$ and
$\deg L_i = e_i$ for $i=1,\ldots,l$. \end{theorem}

\begin{proof}
One can adapt the proof of \cite[Proposition 3.2]{VT2} to get the desired conclusion.
\end{proof}

\begin{remark} \label{regularsequenceremark}
By Theorem \ref{regularsequence}, when $\X$ is an ACM set of
points in $\pnr$, then there exists $L_1,\ldots,L_r$ with $\deg
L_i = e_i$ such that the $L_i$'s give rise to a regular sequence.
After a change of coordinates, we may assume that $L_i =
x_{i,0}$ for $i = 1,\ldots,r$. Note that $x_{i,0}$ does not vanish
at any point of $\X$.  Furthermore, since each $x_{i,0}$ is
homogeneous, any permutation of $\{x_{1,0}, \ldots,x_{r,0}\}$ is
also a regular sequence on $R/I_{\X}$.
\end{remark}

The {\bf coordinate ring of $\X$}, that is, $R/I_{\X}$,  inherits
the $\N^r$-grading of $R$.  We can then apply the following
definition:

\begin{definition} Let $I$ be a multihomogeneous ideal
of $R$. The {\bf Hilbert function of $S = R/I$ } is the numerical
function $H_{S}:\N^r \rightarrow \N$ defined by \[H_S(\ui):=
\dim_{k} S_{\ui} = \dim_{k} R_{\ui} - \dim_{k}(I)_{\ui}~~\mbox{for
all $\ui \in \N^r$.}\] When $S = R/I_{\X}$ is the coordinate ring
of a set of points $\X$, then we usually say $H_S$ is the {\bf
Hilbert function of } $\X$, and write $H_{\X}$.
\end{definition}

\begin{remark} An interesting open problem is to classify
what functions can be the Hilbert function of a set of points in
$\pnr$.  When $r=1$, then the set of valid Hilbert functions for
sets of points in $\pr^n$ was first classified by Geramita,
Maroscia, and Roberts \cite{gmr}. However, very few results are known if $r
\geq 2$.  See \cite{GuMaRa,VT1,VT2} for more on this problem, and
some necessary conditions.\end{remark}

In the study of the Hilbert functions of points in $\pr^n$,
one can use the first difference Hilbert function to ascertain
certain geometric and algebraic information about the set of points.  As we shall
see in the next section, a generalized first difference
Hilbert function is a tool that provides
information about the ACMness of sets of points in $\pnr$.
The definition that we shall require is:

\begin{definition}
If $H:\N^r \rightarrow \N$ is a numerical function,
then the {\bf first difference function of $H$}, denoted $\Delta H$, is defined to be
\[
\Delta H(\ui) := \sum_{\underline{0}~ \preceq~ \underline{l} =
(l_1,\ldots,l_r) \preceq (1,\ldots,1)} (-1)^{|\underline{l}|} H(i_1-l_1,\ldots,
i_r-l_r)
\]
where $H(\uj) = 0$ if $ \uj \not\succeq \underline{0}$ and
$|\underline{l}| = l_1 + \cdots + l_r$.
\end{definition}

Note that when $r=1$, we recover the well known first difference function
$\Delta H(i) = H(i) - H(i-1)$ where $H(i) = 0$ if $i < 0$.


\section{ACM sets of points and their Hilbert functions}

In this section we revisit a classification of ACM sets of points
in $\popo$ due to Giuffrida, Maggioni, and Ragusa \cite{GuMaRa}:

\begin{theorem} \label{GMR-main}
Let $\X \subseteq \popo$ be a finite set of points
with Hilbert function $H_{\X}$.
Then $\X$ is ACM if and only if
$\Delta H_{\X}$ is the Hilbert function
of a bigraded artinian quotient of $k[x_1,y_1]$.
\end{theorem}
\noindent As shown by the second author \cite{VT2}, one direction
of this classification extends quite
naturally to an arbitrary multiprojective space, thus giving
us a necessary condition for a set of points to be ACM. However,
what was not known was whether or not  the converse statement held;
we show via an example that the converse fails, thus showing that
ACM sets of points cannot be classified by Hilbert functions. We
also give a new proof for Theorem \ref{GMR-main}.  We begin
by recalling the partial generalization of Theorem \ref{GMR-main}.

\begin{theorem}[{\cite[Corollary 3.5]{VT2}}]\label{halfACMclassify}
Let $\X \subseteq \pnr$ be a finite set of points with Hilbert function $H_{\X}$.
If $\X$ is ACM, then $\Delta H_{\X}$ is the Hilbert function of
an $\N^r$-graded artinian quotient of
$k[x_{1,1},\ldots,x_{1,n_1},\ldots,x_{r,1},\ldots,x_{r,n_r}]$.
\end{theorem}

\begin{proof} We sketch out the idea of the proof.
Because $\X$ is ACM, by Theorem \ref{regularsequence} there exists
a regular sequence $\overline{L}_1,\ldots,\overline{L}_r \in
R/I_{\X}$ with $\deg L_i = e_i$.  Let $J_0 = I_{\X}$ and $J_i =
(J_{i-1},L_i)$ for $i=1,\ldots,r$. Then, for each $i$ we have a
short exact sequence
\[0 \rightarrow R/J_{i-1}(-e_i) \stackrel{\times \overline{L}_i}{\longrightarrow}
R/J_{i-1} \longrightarrow R/J_i \rightarrow 0.\] Using the $r$
short exact sequences, one can show that
$\Delta H_{\X}$ is the Hilbert function of
$R/J_r$.  Furthermore, because $R/I_{\X}$ is
ACM and $J_r = I_{\X} + (L_1,\ldots,L_r)$, we have that $R/J_r$ is
artinian. By Remark \ref{regularsequenceremark}, if we make a change of coordinates
so that $L_i
= x_{i,0}$, then
\[R/J_r \cong
(R/(x_{1,0},\ldots,x_{r,0}))/((I_{\X},x_{1,0},\ldots,x_{r,0})/(x_{1,0},\ldots,x_{r,0}))
\]
and $R/(x_{1,0},\ldots,x_{r,0}) \cong  k[x_{1,1},\ldots,x_{1,n_1},\ldots,x_{r,1},\ldots,x_{r,n_r}]$.
\end{proof}

The following two examples show that the converse to Theorem \ref{halfACMclassify}
is not true in general;  these examples are the first known counterexamples
to the converse statement.

\begin{example}\label{example1}
Let $P_1,\ldots,P_6$ be six points in general position in $\pr^2$.
By general position we mean that no more than two points lie on
line, and no more than five points lie on a conic. Set $Q_{i,j} := P_i
\times P_j \in \pr^2 \times \pr^2$, and let $\X$ be the following
set of 27 points:
\begin{eqnarray*}
\X &=& \{Q_{1,1},Q_{1,2},Q_{1,3},Q_{1,4},Q_{1,5},Q_{1,6},Q_{2,1},Q_{2,3},Q_{2,4},Q_{2,6}, Q_{3,1},Q_{3,2},Q_{3,5},
Q_{3,6}, \\
&&Q_{4,1},Q_{4,2},Q_{4,5},Q_{4,6}, Q_{5,1},Q_{5,3},Q_{5,6},Q_{6,1},Q_{6,2},Q_{6,3},Q_{6,4},Q_{6,5},Q_{6,6}\}.
\end{eqnarray*}
Then $\X$ is not ACM since $R/I_{\X}$ has projective dimension 5 (and not 4 for
$R/I_{\X}$ to be Cohen-Macaulay)
because the minimal graded resolution is
\[
0 \rightarrow R \rightarrow R^{13} \rightarrow R^{38} \rightarrow
R^{42} \rightarrow R^{17} \rightarrow R \rightarrow R/\Ix \rightarrow 0\] where we have
suppressed the bigraded shifts.
  For this set of points, $H_{\X}$ and $\Delta H_{\X}$ are
\[H_{\X} =
\begin{bmatrix}
1 & 3 & 6 & 6  & \cdots \\
3 & 9 & 18 & 18 & \cdots \\
6 & 18 & 27 & 27 & \cdots \\
6 & 18 & 27 & 27 & \cdots \\
\vdots & \vdots &\vdots & \vdots &\ddots
\end{bmatrix}
~~\text{and}~~ \Delta H_{\X} =
\begin{bmatrix}
1 & 2 & 3 & 0 & \cdots \\
2 & 4 & 6 & 0 & \cdots \\
3 & 6 & 0 & 0 & \cdots \\
0 & 0 & 0 & 0 & \cdots \\
\vdots&\vdots&\vdots&\vdots&\ddots
\end{bmatrix}.\]
(The $(i,j)$th entry of the above matrix corresponds to the value of the Hilbert
function at $(i,j)$,
where we start our indexing at $(0,0)$.)
However, $\Delta H_{\X}$ equals $H_{S/I}$, the Hilbert function of $S/I$ where $S = k[x_1,x_2,y_1,y_2]$
and
\[ I = (x_1,x_2)^3+(y_1,y_2)^3 + (x_1,x_2)^2(y_1,y_2)^2.\]
Note that $S/I$ is artinian since $H_{S/I}(\ui) = 0$ for all but a finite number of $\ui \in \N^2$.
This shows that the converse of Theorem \ref{halfACMclassify} cannot hold because $\Delta H_{\X}$ is the
Hilbert function of a bigraded artinian quotient, but $\X$ is not ACM.
\end{example}

As the following example illustrates, we cannot expect any general
classification of ACM sets of points to be based solely upon the
Hilbert function.

\begin{example}\label{example2}
Let $P_{ij} = [1:i:j] \in \pr^2$, and let $Q_{ijkl}= P_{ij} \times P_{kl} \in \pr^2 \times \pr^2$.
Consider the following 27 points in $\pr^2 \times \pr^2$:
\begin{eqnarray*}
 \Y &= &\{Q_{1121},Q_{1122},Q_{1131},Q_{1221},Q_{1222},Q_{1231},Q_{1321},Q_{1322},Q_{1331}, \\
&& Q_{2111},Q_{2112},Q_{2113}, Q_{2121},Q_{2122},Q_{2131},Q_{2211},Q_{2212},Q_{2213},\\
&&Q_{2221},Q_{2222},Q_{2231}, Q_{3111},Q_{3112},Q_{3113},Q_{3121},Q_{3122},Q_{3131}\}.
\end{eqnarray*}
Using {\tt CoCoA} to compute the resolution of $R/I_{\Y}$ we get
\[0 \rightarrow  R^{12} \rightarrow R^{38} \rightarrow R^{42} \rightarrow R^{17} \rightarrow R \rightarrow R/I_{\Y}
\rightarrow 0,\] where we have suppressed all the bigraded shifts.
So $\Y$ is ACM because the projective dimension is four. If $\X$
is the set of 27 nonACM points from the last example, then
\[H_{\X} = H_{\Y} = \begin{bmatrix}
1 & 3 & 6 & 6  & \cdots \\
3 & 9 & 18 & 18 & \cdots \\
6 & 18 & 27 & 27 & \cdots \\
6 & 18 & 27 & 27 & \cdots \\
\vdots & \vdots &\vdots & \vdots &\ddots
\end{bmatrix}.\]
So ACM and nonACM sets of points can have the exact same Hilbert function.
\end{example}

We can  extend the above examples to show that
the converse of Theorem \ref{halfACMclassify}
fails to hold in any multiprojective space $\pnr$ with $r \geq 2$
and with at least two $n_i$'s greater than or equal to two.

\begin{example}
Consider the multiprojective space $\pnr$ with $r \geq 2$.  Suppose further
that $n_i,n_j \geq 2$ with $i \neq j$.
Set $Q_k = [1:0:\cdots:0] \in \pr^{n_k}$
for $k \in \{1,\ldots,r\}\backslash \{i,j\}$.  For
any point $P =[a_{1}:a_{2}:a_{3}] \in \pr^2$, let $P' = [a_{1}:a_{2}:a_{3}:0:\cdots:0]
\in \pr^{n_i}$, and $P'' = [a_{1}:a_{2}:a_{3}:0:\cdots:0]
\in \pr^{n_j}$.
Let $\X$ be the set of points from Example \ref{example1}.
Consider the following set of points:
\[\X' = \{Q_1\times \cdots \times P'_i \times \cdots \times P''_j \times \cdots \times Q_r ~|~
P_i \times P_j \in \X\} \subseteq \pnr.\] If the point $P_i \times
P_j \in \X \subseteq \pr^2 \times \pr^2$ has defining ideal
$(L_{i,1},L_{i,2},L_{j,1},L_{j,2})$ in
$k[x_0,x_1,x_2,y_0,y_1,y_2]$, then let $L'_{i,t}$, respectively
$L'_{j,t}$, denote the forms we obtain replacing $x_{t}$ with
$x_{i,t}$, respectively, $y_t$ with $x_{j,t}$.   The defining
ideal of $Q_1\times \cdots \times P'_i \times \cdots \times P''_j
\times \cdots \times Q_r \in \X'$ has form
\[(x_{1,1},\ldots,x_{1,n_1},\ldots,L'_{i,1},L'_{i,2},x_{i,3},\ldots,x_{i,n_i},\ldots
L'_{j,1},L'_{j,2},x_{j,3},\ldots,x_{j,n_j},\ldots,x_{r,1},\ldots,x_{r,n_r}).\]
So, we have
\[R/I_{\X'} \cong k[x_{1,0},x_{2,0},\ldots,x_{i,0},x_{i,1},x_{i,2},\ldots,x_{j,0},x_{j,1},x_{j,2},
\ldots,x_{r,0}]/\tilde{I}_{\X}\]
where by $\tilde{I}_{\X}$ we mean the ideal generated by the elements of $I_{\X} \subseteq
k[x_0,x_1,x_2,y_0,y_1,y_2]$, where we replace $x_t$ by $x_{i,t}$ and $y_t$ by
$x_{j,t}$.  The elements $x_{1,0},\ldots,\hat{x}_{i,0},\ldots,\hat{x}_{j,0},\ldots,x_{r,0}$
then form a regular sequence so that
\[ \frac{k[x_{1,0},x_{2,0},\ldots,x_{i,0},x_{i,1},x_{i,2},\ldots,x_{j,0},x_{j,1},x_{j_2},
\ldots,x_{r,0}]}{(\tilde{I}_{\X},x_{1,0},\ldots,\hat{x}_{i,0},\ldots,\hat{x}_{j,0},\ldots,x_{r,0})} \cong
k[x_0,x_1,x_2,y_0,y_1,y_2]/I_{\X}.\]
Then $\X'$ will not be ACM because
\[\operatorname{depth}(R/I_{\X'}) = r-2 +
\operatorname{depth}(k[x_0,x_1,x_2,y_0,y_1,y_2]/I_{\X}) = r-1.\]
However, $\Delta H_{\X'}$ is the Hilbert function
of an artinian quotient since
\[\Delta H_{\X'}(\ui) =
\begin{cases}
\Delta H_{\X}(a,b) & \mbox{if $\ui = ae_i + be_j = (0,\ldots,a,\ldots,b,\ldots,0)$} \\
0 & \mbox{otherwise}
\end{cases}\]
and this function is the Hilbert function of $S/I$ where
\begin{eqnarray*}
I &=& (x_{i,1},x_{i,2})^3+(x_{j,1},x_{j,2})^3 + (x_{i,1},x_{i,2})^2(x_{j,1},x_{j,2})^2 +
(x_{i,3},\ldots,x_{i,n_i}) + \\
&& (x_{j,3},\ldots,x_{j,n_j})+
 S_{e_1} + S_{e_2}+\cdots +\hat{S}_{e_i} + \cdots +\hat{S}_{e_j} + \cdots +S_{e_r}.
\end{eqnarray*}
where $S_{e_l} = (x_{l,1},\ldots,x_{l,n_l})$ and
 $S = k[x_{1,1},\ldots,x_{1,n_1},\ldots,x_{r,1},\ldots,x_{r,n_r}]$.

The set of points $\Y$ in Example \ref{example2} can similarly be extended to a set of points
in $\Y' \subseteq \pnr$ where $H_{\X'}$ and $H_{\Y'}$ are equal.
\end{example}

In light of the above examples, we see that to distinguish ACM
sets of points from nonACM, we will need more information than
just the Hilbert function of the set of points. However, as a
corollary of Theorem \ref{halfACMclassify}, we can eliminate
certain sets of points as being ACM directly from their Hilbert
function. A similar result was also proved in \cite[Theorem
4.7]{SVT}. If $\ui = (i_1,\ldots,i_r), \uj = (j_1,\ldots,j_r) \in
\N^r$, we set $\min\{\ui,\uj\} :=
(\min\{i_1,j_1\},\ldots,\min\{i_r,j_r\}).$

\begin{corollary}  Let $\X \subseteq \pnr$ be a finite set of points with
Hilbert function $H_{\X}$.  If there exists $\ui,\uj \in \N^r$ such
that $H_{\X}(\ui) = H_{\X}(\uj) = |\X|$ but $H_{\X}(\underline{k}) \neq |\X|$
with $\underline{k} = \min\{\ui,\uj\}$,
then $\X$ is not ACM.
\end{corollary}

\begin{proof}  For any $\ui \in \N^r$, the Hilbert function of $\X$ satisfies
\[H_{\X}(\ui) = \sum_{\underline{0} \preceq \uj \preceq \ui}\Delta H_{\X}(\uj).\]
When $\X$ is ACM, by Theorem \ref{halfACMclassify} we have $\Delta H_{\X}(\uj) \geq 0$
for all $\uj \in \N^r$.  But then there exists a $\underline{k} \in \N^r$ such
that $H_{\X}(\ui) =|\X| $ if and only if $\ui \succeq \underline{k}$.
\end{proof}

It is not presently known whether the converse of Theorem \ref{halfACMclassify}
fails to hold in multiprojective spaces of the form $\pnr$ with $r \geq 2$ and with
only one $n_i \geq 1$, and the rest of the $n_j$'s equal to one.
We end this section by giving partial evidence that
the converse of Theorem \ref{halfACMclassify}
may hold for sets of points in $\pr^1 \times \cdots \times \pr^1$
($r \geq 3$ times).
This result will also allow us to give a new proof for
Theorem \ref{GMR-main}.

\begin{theorem}\label{depth=r-1}
Let $\X$ be set of points in $\pr^1 \times \cdots \times \pr^1$ ($r \geq 2$ times)
and suppose that $\depth(R/\Ix) = r-1$.  If $H_{\X}$ is the Hilbert function
of $\X$, then $\Delta H_{\X}$ is not the Hilbert function
of an $\N^r$-graded artinian quotient of $k[x_{1,1},x_{2,1},\ldots,x_{r,1}]$.
\end{theorem}

To prove this statement, we will need the following two technical
lemmas.

\begin{lemma}\label{lemmaA}
Let $S = k[x_{1,1},\ldots,x_{r-1,1},x_{r,0},x_{r,1}]$ be an
$\N^r$-graded ring with $\deg x_{i,1} = e_i$ for
$i=1,\ldots,r-1$, and $\deg x_{r,0} = \deg x_{r,1} = e_r$.  Let
$J \subseteq S$ be any $\N^r$-graded ideal of $S$. If
$H_{S/J}(\ui) \leq H_{S/J}(\ui+e_r)$, then $H_{S/J}(\ui+e_r) =
H_{S/J}(\ui)$ or $H_{S/J}(\ui) +1$.
\end{lemma}

\begin{proof}
We are given that $\dim_k S_{\ui+e_r} - \dim_k (J)_{\ui+e_r}
 \geq  \dim_k S_{\ui} - \dim_k (J)_{\ui}$.
Now, for any $\ui = (i_1,\ldots,i_r) \in \N^r$, $\dim_k S_{\ui} =
i_r+1$. Thus $i_r+2 - \dim_k (J)_{\ui+e_r} \geq i_r+1 -  \dim_k
(J)_{\ui}$. But because $\dim_k (J)_{\ui+e_r} \geq \dim_k
(J)_{\ui}$ for all $\ui$, we must have
\[\dim_k (J)_{\ui} \geq \dim_k (J)_{\ui+e_r}-1 \geq \dim_k (J)_{\ui}-1.\]
So $\dim_k (J)_{\ui+e_r} = \dim_k (J)_{\ui}$ or $\dim_k
(J)_{\ui}+1$, and thus the conclusion follows.
\end{proof}

\begin{lemma}\label{lemmaB}
Let $\X \subseteq \pr^1 \times \cdots \times \pr^1$ be any finite
set of points such that $\depth(R/\Ix) = r-1$.
Suppose $x_{1,0},\ldots,x_{r-1,0}$ is a regular sequence on $R/\Ix$, and suppose that
$\overline{x}_{1,0},\ldots,\overline{x}_{r,0}$ are nonzero divisors on $R/\Ix$. If
\[H_{R/(\Ix,x_{1,0},\ldots,x_{r-1,0})}(\ui) \leq H_{R/(\Ix,x_{1,0},\ldots,
x_{r-1,0})}(\ui+e_r) ~~\mbox{for all $\ui \in \N^r$,}\] then
\[
H_{R/(\Ix,x_{1,0},\ldots,x_{r,0})}(\ui) = \left\{
\begin{array}{ll}
1 & \mbox{if}~~ (I_{\X})_{\ui} = (0)\\
0 & \mbox{if}~~ (I_{\X})_{\ui} \neq (0).
\end{array}
\right.
\]
\end{lemma}

\begin{proof}  To simplify our notation, set
$J = (\Ix,x_{1,0},\ldots,x_{r,0})$. Consider any $\ui \in \N^r$
such that $(\Ix)_{\ui} = (0)$. Then $(J)_{\ui} =
(x_{1,0},\ldots,x_{r,0})_{\ui}$.  Hence $(R/J)_{\ui} =
(R/(x_{1,0},\ldots,x_{r,0}))_{\ui} =
(k[x_{1,1},\ldots,x_{r,1}])_{\ui},$ and thus $H_{R/J}(\ui) =
\dim_k (k[x_{1,1},\ldots,x_{r,1}])_{\ui} = 1$.

Let $\pi_i(\X) = \{R_{i,1},\ldots,R_{i,t_i}\}$ be the set of the
distinct $i$th coordinates of the points that appear in $\X$. If
$L_{R_{i,j}}$ is the form of degree $e_i$ that passes through the
point $R_{i,j}$, then $L_{R_{i,1}}\cdots L_{R_{i,t_i}}$ is a
minimal generator of $\Ix$ of degree $t_ie_i$. Now
$H_{\X}((t_i-1)e_i) = H_{\X}(t_ie_i) = |\pi_i(\X)| = t_i$ (for example,
see \cite[Proposition 4.6]{VT1}).
Because $\overline{x}_{i,0}$ is a nonzero divisor, the short exact sequence
\[0 \rightarrow R/\Ix(-e_i)
\stackrel{\times \overline{x}_{i,0}}{\longrightarrow} R/\Ix
\rightarrow R/(\Ix,x_{i,0}) \rightarrow 0,\] implies that
$H_{R/(\Ix,x_{i,0})}(t_ie_i) = H_{\X}(t_ie_i) - H_{\X}((t-1)e_i)
= 0$. In other words, $R_{t_ie_i} = (\Ix,x_{i,0})_{t_ie_i}$.  So,
for any $\ui \succeq t_ie_i$, $R_{\ui} = (\Ix,x_{i,0})_{\ui}
\subseteq (J)_{\ui}$, and hence $H_{R/J}(\ui) =0$. So,
$H_{R/J}(\ui) = 0$ for all $\ui \succeq t_ie_i$ and each $i =
1,\ldots,r$.

Now consider any $\ui \in \N^r$ such that $\ui \preceq
(t_1-1,\ldots,t_r-1)$ and $(\Ix)_{\ui} \neq (0)$. Assume that
$\ui$ is minimal, i.e., $(\Ix)_{\ui} \neq (0)$, but
$(\Ix)_{\ui-e_j} = (0)$ for $j=1,\ldots,r$. There is then a
minimal generator $F \in \Ix$ of degree $\ui$ which we can write
as
\begin{equation}\label{cform}
F = cx_{1,1}^{i_1}x_{2,1}^{i_2}\cdots x_{r,1}^{i_r} +
\sum_{\mbox{$m$ a monomial of degree $\ui$ in
$(x_{1,0},\ldots,x_{r,0})$}}  c_{m}m.
\end{equation}
If $c \neq 0$, then because $x_{1,0},\ldots,x_{r,0}$ and $F$ are
in $(J)_{\ui}$, we then have $R_{\ui} = (J)_{\ui}$,  from which
it follows that $H_{R/J}(\uj) = 0$ for all $\uj \succeq \ui$.

It thus remains to show that we can find a minimal generator $F$
of degree $\ui$ with form $(\ref{cform})$ and $c \neq 0$. Suppose
not, that is, $c = 0$.  Then $F$ has the form
\begin{eqnarray*}\label{mainF}
F &=& x_{1,0}H_1(x_{1,0},\ldots,x_{r,1}) +
x_{1,1}^{i_1}x_{2,0}H_2(x_{2,0},
\ldots,x_{r,1})+  x_{1,1}^{i_1}x_{2,1}^{i_2}x_{3,0}H_3(x_{3,0},\ldots,x_{r,1})\\
&& + \cdots + x_{1,1}^{i_1}x_{2,1}^{i_2}\cdots
x_{r-1,1}^{i_{r-1}}x_{r,0}H_r(x_{r,0}, x_{r,1}).
\end{eqnarray*}

Set $F' = x_{1,1}^{i_1}x_{2,1}^{i_2}\cdots
x_{r-1,1}^{i_{r-1}}x_{r,0}H_r(x_{r,0}, x_{r,1})$.  We first show
that $F' \neq 0$.  If $F' = 0$, then $F \in
(x_{1,0},\ldots,x_{r-1,0})_{\ui}$, i.e.,
\begin{equation}\label{formF}
F = x_{1,0}G_1 + x_{2,0}G_2 + \cdots + x_{r-1,0}G_{r-1}.
\end{equation}
But then $x_{r-1,0}G_{r-1} \in (\Ix,x_{1,0},\ldots,x_{r-2,0})$,
and since $x_{r-1,0}$ is regular on
$R/(\Ix,x_{1,0},\ldots,x_{r-2,0})$, we have $G_{r-1} \in
(\Ix,x_{1,0},\ldots,x_{r-2,0})_{\ui-e_{r-1}}$.  Because
$(\Ix)_{\uj} = (0)$ for all $\uj \prec \ui$ we must have that
$G_{r-1} \in (x_{1,0},\ldots,x_{r-2,0}).$ So $G_{r-1} =
x_{1,0}G'_1 + \cdots + x_{r-2,0}G'_{r-2}$, and subbing back into
$(\ref{formF})$ we get
\[F = x_{1,0}G_1 + \cdots + x_{r-1,0}( x_{1,0}G'_1 + \cdots
+ x_{r-2,0}G'_{r-2}) = x_{1,0}G''_1 + x_{2,0}G''_2 + \cdots +
x_{r-2,0}G''_{r-2}\] where $G''_i = G_i + x_{r-1,0}G'_i$.
Similarly, we can show that $G''_{r-2} \in
(x_{1,0},\ldots,x_{r-3,0})$, and thus $F = x_{1,0}E_1+\cdots
+x_{r-3,0}E_{r-3}$ for some appropriate forms $E_i$.  We can
continue this process to eventually show that $F$ is divisible by
$x_{1,0}$, that is, $F = x_{1,0}F_1$.  But since $x_{1,0}$ is a
regular on $R/\Ix$, this implies that $F_1 \in (\Ix)$,
contradicting the minimality of the degree of $F$.  So $F' \neq
0$.

Set $J' = (I_{\X},x_{1,0},\ldots,x_{r-1,0})$. We now claim that
$H_{R/J'}(\ui) = i_r$.  Now $H_{R/J'}(\ui-e_r) =
H_{R/(x_{1,0},\ldots,x_{r-1,0})}(\ui-e_r) = i_r$. Our hypotheses
then imply that $i_r = H_{R/J'}(\ui-e_r) \leq H_{R/J'}(\ui)$. But
because
\[R/J' \cong
\frac{R/(x_{1,0},\ldots,x_{r-1,0})}{J'/(x_{1,0},\ldots,x_{r-1,0})}
\cong k[x_{1,1},\ldots,x_{r-1,0},x_{r,0},x_{r,1}]/L ~~\mbox{for some ideal $L$},\] Lemma
\ref{lemmaA} implies $H_{R/J'}(\ui) = H_{R/J'}(\ui-e_r)$ or
$H_{R/J'}(\ui-e_r)+1$.
If $H_{R/J'}(\ui) = i_r+1$, then
\begin{eqnarray*}
\dim_k (J')_{\ui} &=& (i_1+1)\cdots (i_r+1) - (i_r+1) = \dim_k
(x_{1,0},\ldots,x_{r-1,0})_{\ui}.
\end{eqnarray*}
This means that $(J')_{\ui} = (x_{1,0},\ldots,x_{r-1,0})_{\ui}$,
and hence $F \in  (x_{1,0},\ldots,x_{r-1,0})_{\ui}$.  But as
shown above, $F' \neq 0$, so $F \not\in
(x_{1,0},\ldots,x_{r-1,0})_{\ui}$. Thus  $H_{R/J'}(\ui) =i_r$,
and $\dim_k (J')_{\ui} = (i_1+1)\cdots (i_r+1) - i_r$.

Let $M_{\ui}$ denote the set of all monomials of degree $\ui$ in
$R$. A basis for $(J')_{\ui}$ is then
\[\{F'\} \cup  (M_{\ui} \setminus
\{x_{1,1}^{i_1}x_{2,1}^{i_2}\cdots x_{r-1,1}^{i_{r-1}}m ~|~ m =
x_{r,0}^{i_r-a}x_{r,1}^{a} ~\mbox{for $a = 0,\ldots,i_r$}\}).\]
To see that this is a basis, note that the elements are linearly
independent and we have $\dim_k (J')_{\ui}$ elements. Then, for
any $b \geq 1$, the following set of elements in
$(J')_{\ui+be_r}$ is linearly independent:
\begin{eqnarray*}
\mathcal{B}_b &=& \{x_{r,0}^bF',x_{r,0}^{b-1}x_{r,1}F',\ldots,
x_{r,1}^{b}F'\}\\
&& \cup ~(M_{\ui+be_r} \setminus
\{x_{1,1}^{i_1}x_{2,1}^{i_2}\cdots x_{r-1,1}^{i_{r-1}}m ~|~ m =
x_{r,0}^{i_r+b-a}x_{r,1}^{a} ~\mbox{for $a = 0,\ldots,i_r+b$}\}).
\end{eqnarray*}
Note that
\begin{eqnarray*}
|\mathcal{B}_b| &=& (i_1+1)\cdots (i_{r-1}+1)(i_r+b+1) - i_r.
\end{eqnarray*}
Now, because $i_r = H_{R/J'}(\ui) \leq H_{R/J'}(\ui+be_j)$ we have
\[\dim_k R_{\ui+be_j} - \dim_k (J')_{\ui+be_j} \geq i_r.\]
Hence
\[|\mathcal{B}_b| =
(i_1+1)\cdots (i_{r-1}+1)(i_r+b+1) -i_r = \dim_k R_{\ui+be_j} -
i_r \geq \dim_k (J')_{\ui+be_j} \geq |\mathcal{B}_b|.\] It
follows that $\dim_k (J')_{\ui+be_j} = |\mathcal{B}_b|$, and hence
the elements of $\mathcal{B}_b$ form a basis for $(J')_{\ui+be_j}$.

We now pick $p$ so that $i_r+p = t_r = |\pi_r(\X)|$. (We have $p
\geq 1$ since $\ui \preceq (t_1-1,\ldots,t_r-1)$, i.e., $i_r <
t_r$.) As noted, $L_{R_{r,1}}\cdots L_{R_{r,t_r}} \in
(\Ix)_{t_re_r}$, and furthermore, each $L_{R_{r,i}}$ has form
$b_{i,0}x_{r,0} + b_{i,1}x_{r,1}$ with $b_{i,1} \neq 0$ for all
$i$ because $x_{r,0}$ is not a zero-divisor, i.e.,
no point $R_{r,i}$ has form $[0:1]$.  Now
\[x_{1,1}^{i_1}x_{2,1}^{i_2}\cdots x_{r-1,1}^{i_{r-1}}
L_{R_{r,1}}\cdots L_{R_{r,t_r}} \in (\Ix)_{\ui+pe_r} \subseteq
(J')_{\ui+pe_r},\] so $x_{1,1}^{i_1}x_{2,1}^{i_2}\cdots
x_{r-1,1}^{i_{r-1}} L_{R_{r,1}}\cdots L_{R_{r,t_r}} $ can be
written as a linear combination of the elements of
$\mathcal{B}_p$.  But this cannot happen because
$x_{1,1}^{i_1}x_{2,1}^{i_2}\cdots x_{r-1,1}^{i_{r-1}}
L_{R_{r,1}}\cdots L_{R_{r,t_r}} $ contains the term
$x_{1,1}^{i_1}\cdots x_{r-1,1}^{i_{r-1}}x_{r,1}^{i_r+p}$, but
this term does not appear in any of our basis elements.  So, if
$F$ is a minimal generator of degree $\ui$, it must have the form
$(\ref{cform})$ with $c \neq 0$.
\end{proof}

\begin{example}
The hypothesis that
$H_{R/(\Ix,x_{1,0},\ldots,x_{r-1,0})}(\ui) \leq
H_{R/(\Ix,x_{1,0},\ldots,x_{r-1,0})}(\ui+e_r)$
in the above statement is necessary.
For example, in $\popo$ consider the set of
points $\X=\{P_{1,1},P_{2,2},P_{3,3}\}$
where the defining ideal of $I_{P_{i,i}}=(x_1-ix_0,y_1-iy_0)$.
Then $\depth(R/\Ix) = 1$, $x_0$ is a regular sequence on $R/\Ix$,
and $\overline{x}_0$ and $\overline{y}_0$ are nonzero-divisors on $R/\Ix$.  We have
\[H_{\X} =
\begin{bmatrix}
1 & 2 & 3 & 3&\cdots \\
2 & 3 & 3 & 3& \cdots \\
3 & 3 & 3 & 3&\cdots \\
3 & 3 & 3 & 3&\cdots \\
\vdots & \vdots &\vdots & \vdots &\ddots
\end{bmatrix}
~~\text{and thus}~~
H_{R/(\Ix,x_0)} =
\begin{bmatrix}
1 & 2 & 3 & 3&\cdots \\
1 & 1 & 0 & 0& \cdots \\
1 & 0 & 0 & 0&\cdots \\
0 & 0 & 0 & 0&\cdots \\
\vdots & \vdots &\vdots & \vdots &\ddots
\end{bmatrix}.\]
The function $H_{R/(\Ix,x_0)}$ fails to have
$H_{R/(\Ix,x_0)}(i,j) \leq H_{R/(\Ix,x_0)}(i,j+e_2)$ for all
$(i,j) \in \N^2$.  Now $\dim_k (\Ix)_{1,1} \neq 0$ but
\[
H_{R/(I_{\X},x_0,y_0)} =
\begin{bmatrix}
1 & 1 & 1 & 0 & \cdots \\
1 & 1 & 0 & 0 & \cdots \\
1 & 0 & 0 & 0 & \cdots \\
0 & 0 & 0 & 0 & \cdots \\
\vdots&\vdots&\vdots&\vdots&\ddots
\end{bmatrix},\]
that is, $H_{R/(I_{\X},x_0,y_0)} =1$.
\end{example}

\begin{proof}(of Theorem \ref{depth=r-1})
Let $R =
k[x_{1,0},x_{1,1},x_{2,0},x_{2,1},\ldots,x_{r,0},x_{r,1}]$ be the
$\N^r$-graded coordinate ring of $\pr^1 \times \cdots \times
\pr^1$. Since $\depth(R/\Ix) =r-1$, by Theorem
\ref{regularsequence} we can find a regular sequence
$L_1,\ldots,L_{r-1}$ on $R/\Ix$ with $\deg L_i = e_i$. Moreover,
by Remark \ref{regularsequenceremark}, we can assume $L_i =
x_{i,0}$ for $i =1,\ldots,r-1$.  Note that we can also assume
that $x_{r,0}$ is a nonzero divisor of $R/\Ix$.

Set $J_0 = \Ix$, and for $i=1,\ldots,r-1$, set $J_i =
(J_{i-1},x_{i,0})$. Then, for each $i = 1,\ldots,r-1$ we have a
short exact sequence
\[0 \rightarrow R/J_{i-1}(-e_i) \stackrel{\times \overline{x}_{i,0}}{\longrightarrow}
R/J_{i-1} \rightarrow R/J_i \rightarrow 0.\] We thus have
\[H_{R/J_{r-1}}(\ui) = H_{R/J_{r-2}}(\ui) - H_{R/J_{r-2}}(\ui-e_{r-1}) ~~\mbox{for all
$\ui \in \N^r$}\] with $H_{R/J_{r-2}}(\ui) = 0$ if $\ui
\not\succeq \underline{0}$.

For all $\ui \in \N^r$, we have the following exact sequence of
vector spaces:
\begin{equation}    \label{exactsequence}
0 \rightarrow (\ker \times \overline{x}_{r,0})_{\ui} \rightarrow
(R/J_{r-1})_{\ui} \stackrel{\times
\overline{x}_{r,0}}{\rightarrow} (R/J_{r-1})_{\ui+e_r}
\rightarrow (R/(J_r,x_{r,0}))_{\ui+e_r} \rightarrow 0
\end{equation}
where $\overline{x}_{r,0} \neq 0$ is a degree $e_r$ element in
$R/J_{r-1} = R/(\Ix,x_{1,0},\ldots,x_{r-1,0})$.  Because
$\depth(R/\Ix) = r-1$, then there is at least one $\ui\in \N^r$
such that the map
\[
(R/J_{r-1})_{\ui} \stackrel{\times
\overline{x}_{r,0}}{\rightarrow} (R/J_{r-1})_{\ui+e_r}
\]
has a non-zero kernel.  This follows from the fact that
$\overline{x}_{r,0}$ must be a zero divisor of $R/J_{r-1}$, so
there exists a non-zero element $\overline{F} \in
(R/J_{r-1})_{\ui}$ such that $Fx_{r,0} \in J_{r-1}$.

It will now suffice to show that there exists an $\ui \in \N^r$
such that
\[
H_{R/J_{r-1}}(\ui) > H_{R/J_{r-1}}(\ui+e_r).
\]
Indeed, because $H_{R/J_{r-1}}(\ui) =
H_{R/J_{r-2}}(\ui)-H_{R/J_{r-2}}(\ui-e_{r-1})$, the above
inequality would imply
\[H_{R/J_{r-2}}(\ui+e_r) - H_{J_{r-2}}(\ui-e_{r-1}+e_r)-H_{R/J_{r-2}}(\ui) + H_{R/J_{r-2}}(\ui-e_{r-1}) < 0  \Leftrightarrow
\Delta H_{\X}(\ui+e_r) < 0.\] So, not only is $\Delta H_{\X}$
not  the Hilbert function of an artinian quotient, $\Delta
H_{\X}$ is not the Hilbert function of {\it any} quotient because
it has a negative entry.

Suppose, for a contradiction, that $H_{R/J_{r-1}}(\ui) \leq
H_{R/J_{r-1}}(\ui+e_r)$ for all $\ui \in \N^r$.  Because
\begin{eqnarray*}
R/J_{r-1} &\cong &\frac{R/(x_{1,0},\ldots,x_{r-1,0})}
{J_{r-1}/(x_{1,0},\ldots,x_{r-1,0})} \cong
k[x_{1,1},x_{2,1},\ldots,x_{r-1,1},x_{r,0},x_{r,1}]/L
\end{eqnarray*}
with $L \cong J_{r-1}/(x_{1,0},\ldots,x_{r-1,0})$,
$H_{R/J_{r-1}}$ has the Hilbert function of some quotient of
$k[x_{1,1},\ldots,x_{r-1,1},x_{r,0},x_{r,1}]$ with $\deg x_{i,1}
= e_i$ and $\deg x_{r,0} = \deg x_{1,1} = e_r$.  It then follows
by  Lemma \ref{lemmaA} that if  $H_{R/J_{r-1}}(\ui) \leq
H_{R/J_{r-1}}(\ui+e_r)$, then $H_{R/J_{r-1}}(\ui+e_r) =
H_{R/J_{r-1}}(\ui) + 1,$ or $H_{R/J_{r-1}}(\ui+e_r) =
H_{R/J_{r-1}}(\ui).$ We now show that both cases imply $\dim_k
(\ker \times \overline{x}_{r,0})_{\ui} = 0$ for all $\ui \in
\N^r$. \vspace{.15cm}

\noindent {\it Case 1. } If $H_{R/J_{r-1}}(\ui+e_r) =
H_{R/J_{r-1}}(\ui) +1$, then from the exact sequence
(\ref{exactsequence}), we have
\begin{eqnarray*}
\dim_k (\ker \times \overline{x}_{r,0})_{\ui}
& = & H_{R/J_{r-1}}(\ui) - H_{R/J_{r-1}}(\ui+e_r) + H_{R/(J_{r-1},x_{r,0})}(\ui+e_r) \\
&= &-1 + H_{R/(J_{r-1},x_{r,0})}(\ui+e_r).
\end{eqnarray*}
By Lemma \ref{lemmaB}, for any $\ui \in \N^r$ we must have
\[H_{R/(J_{r-1},x_{r,0})}(\ui) =
\dim_k (R/(\Ix,x_{1,0},\ldots,x_{r,0}))_{\ui} = \mbox{$0$ or
$1$}.\] So, if  $H_{R/J_{r-1}}(\ui+e_r) = H_{R/J_{r-1}}(\ui) +1
$, then $\dim_k (\ker \times \overline{x}_{r,0})_{\ui} = 0$
since dimension must be nonnegative.

\noindent {\it Case 2. } If $H_{R/J_{r-1}}(\ui+e_r) =
H_{R/J_{r-1}}(\ui)$, then from (\ref{exactsequence}) we deduce
that
\begin{eqnarray*}
\dim_k (\ker \times \overline{x}_{r,0})_{\ui}
& = & H_{R/J_{r-1}}(\ui) - H_{R/J_{r-1}}(\ui+e_r) + H_{R/(J_{r-1},x_{r,0})}(\ui+e_r) \\
&= &0 + H_{R/(J_{r-1},x_{r,0})}(\ui+e_r).
\end{eqnarray*}
Now $H_{R/J_{r-1}}(\ui) = H_{R/J_{r-1}}(\ui+e_r)$ can occur only
if $H_{\X}(\ui+e_r) < (i_1+1)\cdots (i_r+2)$, that is,
if $(\Ix)_{\ui+e_r} \neq (0)$.  By Lemma \ref{lemmaB} we have
$H_{R/(J_{r-1},x_{r,0})}(\ui+e_r) = 0$, and hence $\dim_k (\ker
\times \overline{x}_{r,0})_{\ui} =0$.

We now see from both Cases 1 and 2, that if $H_{R/J_{r-1}}(\ui)
\leq H_{R/J_{r-1}}(\ui+e_r)$ for all $\ui \in \N^r$, then we must
always have $\dim_k (\ker \times \overline{x}_{r,0})_{\ui} = 0$.
But this contradicts the fact that since $\X$ is not ACM, there
exists some $\ui \in \N^r$ with $\dim_k (\ker \times
\overline{x}_{r,0})_{\ui} > 0$.  So, $H_{R/J_{r-1}}(\ui) >
H_{R/J_{r-1}}(\ui+e_r)$ for some $\ui$, as desired.
\end{proof}

Combining the above result with  Theorem \ref{halfACMclassify}
gives a new proof for Theorem \ref{GMR-main}.

\begin{proof}(Proof of Theorem \ref{GMR-main})  If $\X$ is
a finite set of points in $\popo$, then $\depth(R/\Ix) = 2$ or $1$ by
Theorem \ref{depthdim}.  If $\depth(R/\Ix) =2$, then $\X$ is
ACM, and thus by Theorem \ref{halfACMclassify} we have that
$\Delta H_{\X}$ is the Hilbert function of a bigraded artinian
quotient of $k[x_1,y_1]$.  If $\depth(R/\Ix) = 1$, then $\X$
is not ACM.  If we now apply Theorem \ref{depth=r-1} to
the case $r=2$, we have
that $\Delta H_{\X}$ is not the Hilbert function of a bigraded
artinian quotient.
\end{proof}

\begin{remark} Let $\X$ be a finite set of points in $\popo\times\pr^1$.
Then $\depth(R/\Ix) = 1,2,$ or $3$.  In light of Theorem \ref{depth=r-1}, to
show that the converse
of Theorem \ref{halfACMclassify} holds in $\popo \times \pr^1$, it suffices
to show that if $\depth(R/\Ix) =1$, then $\Delta H_{\X}$ is not
the $\N^3$-graded Hilbert function of an artinian quotient
of $k[x_{1,1},x_{2,1},x_{3,1}]$.
We are currently exploring the case of points in $\pr^1 \times \popo$.
\end{remark}

\begin{example}
In the proof of Theorem \ref{depth=r-1}, to show that
$\Delta H_{\X}$ is not the Hilbert function of an artinian ring, we show
that
$\Delta H_{\X}$ {\it must} have a negative entry.  This approach,
however, will not work for points in $\pr^1 \times \pr^n$ with
$\depth(R/\Ix) =1$ and $n >1$.  For example
let $\X \subseteq \pr^1 \times \pr^2$ where $\X$ is the following
11 points:
\begin{eqnarray*}
\X &=& \{P_1 \times Q_1,P_1\times Q_2,P_1\times Q_3, P_1 \times Q_4,
P_2 \times Q_1,P_2\times Q_2,P_2\times Q_3, P_2 \times Q_5, \\
&&P_3 \times Q_1,P_3\times Q_2,P_3\times Q_3\}
\end{eqnarray*}
where $P_i = [1:i] \in \pr^1$ for $i=1,\ldots,3$,
and $Q_1 = [1:1:1], Q_2 = [1:1:2], Q_3 = [1:1:3], Q_4 = [1:0:1]$
and $Q_5 = [1:0:2]$ in $\pr^2$.
Using {\tt CoCoA} one finds that $R/\Ix$ is not CM.
The Hilbert function of $\X$ is
\[H_{\X} =
\begin{bmatrix}
1 & 3 & 5 & 5 & \cdots \\
2 & 6 & 8 & 8 & \cdots \\
3 & 8 & 11 & 11 & \cdots \\
3 & 8 & 11 & 11 & \cdots \\
\vdots& \vdots & \vdots &\vdots & \ddots
\end{bmatrix}
~~\mbox{and}~~\Delta H_{\X} =
\begin{bmatrix}
1 & 2 & 2 & 0 & \cdots \\
1 & 2 & 0 & 0 & \cdots \\
1 & 1 & 1 & 0 & \cdots \\
0 & 0 & 0 & 0 & \cdots \\
\vdots& \vdots & \vdots &\vdots & \ddots
\end{bmatrix}.
\]
Note that $\Delta H_{\X}$ has no negative values, but $\Delta H_{\X}$
still cannot be the Hilbert function of an artinian quotient.
This is because $\Delta H_{\X}(1,2) = 0$, so we should have
$\Delta H_{\X}(i,j) = 0$ for all $(i,j) \succeq (1,2)$, but
$\Delta H_{\X}(2,2) =1$.
\end{example}


\section{ACM sets of points and their geometry}

In \cite{Gu1} and \cite{VT}, the two authors classified
ACM sets of points in $\popo$ via the geometry of the points.
In this section we revisit this classification; we show
that this geometric criterion extends to
a sufficient condition for ACM sets of points in $\pr^1 \times
\pr^n$.  However, we give an example to show that this criterion fails
to be a necessary condition for a set of ACM points in the
general case.

We begin by adapting the construction and main result of
\cite{Gu2} to reduced points. Let $\X$ be a finite set of
points in $\popo$. Let $\pi_1(\X) = \{P_1,\ldots,P_r\}$,
respectively, $\pi_2(\X) = \{Q_1,\ldots,Q_t\}$ be the set of
first, respectively second, coordinates of the points in $\X$.
For each tuple $(i,j)$ with $1 \leq i \leq r\,, 1 \leq j \leq t$,
set $P_{i,j} = P_i \times Q_j$.  Then, for each such $(i,j)$ define
\[p_{ij}:= \left\{\begin{array}{ll}
1 & \textrm {if } P_{i,j}\in \X \\
0 & \mbox{otherwise.}
\end{array}
\right.\]
The set $\Ssx$ is then defined to be the set of $t$-tuples
\[
\Ssx:=\{(p_{11},\ldots,p_{1t}),(p_{21},\ldots,p_{2t}),\ldots,
(p_{r1},\dots,p_{rt})\}.
\]

With this notation the main result of \cite{Gu2} for distinct
points becomes:

\begin{theorem}[{\cite[Theorem 2.1]{Gu2}}]   \label{acm}
Let $\X$ be a finite set of points in $\popo$.  Then $\X$ is ACM if and only
if the set
$\mathcal S_{\X}$ is a totally ordered set with respect to $\succeq$.
\end{theorem}

We now introduce a geometric condition on a set of points in
$\pr^n\times\pr^m$:

\begin{definition}
Let $\X$ be any finite set of points in $\pr^n\times\pr^m$.  We say that
$\X$ satisfies {\bf property $(\star)$} if whenever $P_1 \times Q_1$ and $P_2 \times
Q_2$ are two points in $\X$ with $P_1 \neq P_2$ and $Q_1 \neq Q_2$,
then either $P_1 \times Q_2 \in \X$ or $P_2 \times Q_1 \in \X$ (or both)
are in $\X$.
\end{definition}

\begin{theorem}\label{theorem*}
Let $\X$ be a finite set of points in $\pr^1 \times \pr^1$.  Then
$\X$ is ACM if and only if $\X$ satisfies property $(\star)$.
\end{theorem}

\begin{proof}  A straightforward exercise will show that the condition
$(\star)$ is equivalent to the condition that $\Ssx$ is totally ordered.
Then apply Theorem \ref{acm}.
\end{proof}

\begin{example}
The simplest example of a nonACM set of points in $\popo$ are two
non-collinear points.  That is, $\X = \{P_1 \times P_1, P_2 \times P_2\}$
where $P_1,P_2$ are two distinct points in $\pr^1$.  Then $\X$ clearly
does not satisfy property $(\star)$.  In this case $\Ssx = \{(1,0),(0,1)\}$ which
is not totally ordered with respect to $\succeq$.
\end{example}

One direction of the above result holds more generally in
$\pr^1\times \pr^n$:

\begin{theorem} \label{geometriccondition}
Let $\X$ be a finite set of points in $\pr^1 \times \pr^n$.
If $\X$ satisfies property $(\star)$, then $\X$ is ACM.
\end{theorem}

\begin{proof}
Let $\X_P$ denote the subset of points in $\X$ whose
first coordinate is $P$ with $P \in \pi_1(\X)$.

\noindent
{\it Claim.}  There exists a point $P \in \pi_1(\X)$ such that $\pi_2(\X_P) = \pi_2(\X)$.

\noindent {\it Proof of Claim.}  We always have $\pi_2(\X_P)
\subseteq \pi_2(\X)$.  Let $P$ be a point of $\pi_1(\X)$ with
$|\pi_2(\X_P)|$ maximal.  We will show that this is the desired
point.
Suppose
there is $Q \in \pi_2(\X) \backslash \pi_2(\X_P)$.  So, there exists a $P \neq P' \in \pi_1(\X)$
such that $P' \times Q \in \X$.  Let $Q' \in \pi_2(\X_P)$ be any point.  So
$P \times Q'$ and $P' \times Q$ are points in $\X$.  By the hypotheses,
$P \times Q$ or $P' \times Q'$ are in $\X$.  But $P \times Q \not\in \X$ (else
$Q \in \pi_2(\X_P)$).  So, for each $Q' \in \pi_2(\X_P)$, $P' \times Q' \in \X$.
But this means $|\pi_2(\X_{P'})| > |\pi_2(\X_P)|,$ contradicting the maximality
of $|\pi_2(\X_P)|$. \hfill$\Box$

We now
prove the statement by induction on $|\pi_1(\X)|$.  If $|\pi_1(\X)|=1$, then
$\X$ is ACM.  To see this, note that $\Ix = I_P + I_{\pi_2(\X)}$ where
$I_P$ is the defining ideal of $P \in \pr^1$ in $R_1 = k[x_0,x_1]$,
but viewed as an ideal of $R =k[x_0,x_1,y_0,\ldots,y_n]$ and $I_{\pi_2(\X)}$
is the defining ideal of $\pi_2(\X) \subseteq \pr^n$ in $R_2= k[y_0,\ldots,y_n]$,
but viewed as an idea of  $R$.  So, the resolution of $R/\Ix \cong R_1/I_P \otimes_k
R_2/I_{\pi_2(\X)}$ is obtained by tensoring together the resolutions of
$R_1/I_P$ and $R_2/I_{\pi_2(\X)}$.  From this resolution we can obtain
the fact that $\X$ is ACM.

For the induction step, set $\Y = \X \backslash \X_{P}$, where $P$
is the point from the claim,
and thus $I_{\X} = I_{\X_P} \cap I_{\Y}$.  Note that $\Y$ also
satisfies $(\star)$, so by induction $\Y$ and $\X_P$ are ACM.
We have a short exact sequence
\[0 \rightarrow R/I_{\X} \rightarrow R/I_{\X_P} \oplus R/I_{\Y}
\rightarrow R/(I_{\Y} + I_{\X_P}) \rightarrow 0.\]
By induction
$R/I_{\X_P}$ and $R/I_{\Y}$ are CM of dimension 2.  It suffices to
show that $R/(I_{\Y}+I_{\X_P})$ is CM of dimension 1.  It then
follows that $R/I_{\X}$ is CM of dimension 2, i.e., $\X$ is ACM.

To prove this, we use the observation from above that $I_{\X_P} =
I_P + I_{\pi_2(\X_P)}$. Now any $G \in I_{\pi_2(\X_P)}$ is also
in $I_{\Y}$ since for any point $P' \times Q' \in \Y$, $Q' \in
\pi_2(\Y) \subseteq \pi_2(\X) = \pi_2(\X_P)$, and hence $G(P'
\times Q') = 0$.  Thus
\[I_{\Y} + I_{\X_P} = I_{\Y} + I_P.\]
Now, by change of coordinates, we can assume $I_P =
(x_0)$.  Also, we can assume that $x_0$ does not pass
through any points of $\pi_1(\Y)$.  So, $x_0$ is a nonzero divisor
of $R/I_{\Y}$.    To finish the proof
we note that by induction, $R/I_{\Y}$ is ACM of dimension 2, and
since $x_0$ is a nonzero divisor of $R/I_{\Y}$, we have
\[R/(I_{\Y}+I_{\X_P}) = R/(I_{\Y},x_0)\]
is CM of dimension 1.  The desired conclusion now holds.
\end{proof}

\begin{remark}
By interchanging the roles of the $x_i$'s and $y_i$'s in the above proof,
the conclusion of the previous theorem also holds for points in $\pr^n \times \pr^1$.
\end{remark}

\begin{remark} In trying
to generalize the above result to points in $\pr^m \times \pr^n$ we ran into the
following difficulty.  We still
have $I_{\Y} + I_{\X_P} = I_{\Y} + I_P$ where $P \in \pr^m$.  By changing
coordinates, we can take $I_P = (x_0,\ldots,x_{m-1})$, and we can assume
that $x_0$ does not pass through any points of $\pi_1(\Y)$.
So, $x_0$ is a nonzero divisor of $R/I_{\Y}$.  So, we know that
$R/(I_{\Y},x_0)$ is CM of dimension 1.   However, we were left
with the question of whether $R/(I_{\Y},x_0,\ldots,x_{m-1})$ is also CM if
$R/(I_{\Y},x_0)$ is CM.  Computer experimentation suggests a positive
answer to this question under the hypotheses of Theorem
\ref{geometriccondition}, thus suggesting Theorem \ref{geometriccondition} may hold
more generally for sets of points in $\pr^m \times \pr^n$.
\end{remark}

\begin{corollary}\label{nonacm}
Suppose $\X \subseteq \pr^1 \times \pr^n$ and $\X$ is not ACM.  Then
there exists a pair of points $P_1 \times Q_1$ and $P_2 \times Q_2 \in \X$ with
$P_1 \neq P_2$, $Q_1 \neq Q_2$, but $P_1 \times Q_2, P_2 \times Q_1 \not \in \X$.
\end{corollary}

While the converse of  Theorem \ref{geometriccondition} holds in $\popo$,
it fails in general.

\begin{example}\label{example3}
Let $P_i = [1:i] \in \pr^1$ for $i=1,\ldots,6$, and let
$P'_1,\ldots,P'_6$ be six points in general position
in $\pr^2$.  Set $Q_{i,j} = P_i \times P'_j \in \pr^1 \times \pr^2$.  Consider
the following set of 27 points:
\begin{eqnarray*}
\X &= & \{Q_{1,1},Q_{1,2},Q_{2,1},Q_{2,2},Q_{2,3},Q_{2,4},Q_{2,5},Q_{3,1},Q_{3,2},Q_{3,3},Q_{3,4},
Q_{3,5}, Q_{4,1},Q_{4,2},\\
&&Q_{4,4},Q_{4,5},Q_{5,1},Q_{5,2},Q_{5,3},Q_{5,4},Q_{5,6},Q_{6,1},Q_{6,2},
Q_{6,3},Q_{6,4},Q_{6,5},Q_{6,6}\}.
\end{eqnarray*}
Using {\tt CoCoA} to compute the resolution, we find
\[0 \rightarrow R^5 \rightarrow R^{13} \rightarrow R^9
\rightarrow R \rightarrow R/\Ix \rightarrow 0. \]
So $\X$ is ACM since the projective dimension is $3$ and $\dim
R = 5$, so by the Auslander-Buchsbaum formula, $\depth(R/\Ix) =2$.
But $\X$ fails property $(\star)$ since $Q_{4,5}$ and $Q_{5,3}$ are in
$\X$, but neither $Q_{4,3}$ nor $Q_{5,5}$ are in $\X$.
\end{example}

We end this section by describing a simple construction
to make sets of points that satisfy property $(\star)$.

\begin{definition}
\label{partitiondefinition} A tuple $\lambda =
(\lambda_1,\ldots,\lambda_r)$ of positive integers is a {\bf
partition} of an integer $s$ if $\sum \lambda_i = s$ and
$\lambda_i \geq \lambda_{i+1}$ for every $i$.  We write $\lambda =
(\lambda_1,\ldots,\lambda_r) \vdash s$.  To any partition $\lambda
\vdash s$ we can associate the
following diagram:  on an $r \times \lambda_1$ grid, place
$\lambda_1$ points on the first line, $\lambda_2$ points on the
second, and so on.  The resulting diagram is called the {\bf
Ferrer's diagram} of $\lambda$.
\end{definition}

\begin{construction}
Let $\lambda = (\lambda_1,\ldots,\lambda_r) \vdash s$, and let
$P_1,\ldots,P_r$ be $r$ distinct points
in $\pr^1$ and $Q_1,\ldots,Q_{\lambda_1}$ be $\lambda_1$
distinct points in $\pr^n$.  Let $\X_{\lambda}$
then be the $s$ points of $\pr^1 \times \pr^n$ where
\begin{eqnarray*}
\X_{\lambda}&=&\{P_1\times Q_1,P_1\times Q_2,\ldots,P_1 \times Q_{\lambda_1},
P_2 \times Q_1, \ldots,P_2 \times Q_{\lambda_2}, \ldots\\
& & P_r \times Q_1,\ldots, P_r \times Q_{\lambda_r}\}.
\end{eqnarray*}
The set of points $\X_{\lambda}$ then resembles a Ferrer's
diagram of $\lambda$ and satisfies property $(\star)$.  By
Theorem \ref{geometriccondition}:
\begin{theorem}\label{Ferrer}
With the notation as above, $\X_{\lambda}$ is ACM.
\end{theorem}
\end{construction}


\section{ACM sets of points and their separators}

In this section we study ACM set of points using the notion of a
separator.  Separators for points in $\pr^n$ were first
introduced by  Orecchia \cite{O} and their properties were later
studied in \cite{abm,b,b2,m1}, to name but a few references.
Separators were recently defined in a multigraded setting by
Marino in \cite{m1,m2} for the special case of points in
$\popo$.   In particular, Marino classified ACM sets of points in
$\popo$ using separators;  we extend some of these ideas in this
section.

We begin by introducing some notation.  If
$S \subseteq \N^r$ is a subset, then $\min S$ is the set of the
minimal elements of $S$ with respect to the partial ordering $\succeq$. For
any $\ui\in \N^r$, define $D_{\ui}:=\{\uj \in \N^r ~|~ \uj \succeq
\ui\}.$ For any finite set $S=\{\underline{s}_1,\ldots,\underline{s}_{p}\}\subseteq \N^r$,
we set \[D_S:=\bigcup_{\underline{s}\in S} D_{\underline{s}}.\]
Note that $ \min D_S = S$; thus $D_S$ can be viewed as the largest
subset of $\N^r$ whose minimal elements are the elements of $S$.

\begin{definition} Let $\X$ be a set of distinct points in $\pnr$
and $P \in \X$. We say that the multihomogeneous form $F\in R$ is
a {\bf separator for} $P$ if $F(P)\neq 0$ and $F(Q)=0$ for all $Q
\in \X \setminus\{P\}$.   We will call $F$ a {\bf minimal
separator} for $P$ if there does not exist a separator $G$ for $P$
with $\deg G \prec \deg F$.

\end{definition}

\begin{definition} Let $\X$ be a set of distinct points in $\pnr$.
Then the {\bf degree of a point} $P\in \X$ is the set
\begin{eqnarray*}
\deg_{\X}(P)&:=&\min\{ \deg F ~|~ F ~~\text{is a separator for $P\in \X$}\}\\
&=&\{\deg F ~|~ \mbox{$F$ is a minimal separator of $P \in \X$}\}.
\end{eqnarray*}
Here, we are using the partial order $\succeq$ on $\N^r$.
\end{definition}

If $\X \subseteq \pr^n$, then $\N$ is a totally ordered set, so we
can talk about {\it the} degree of a point $P \in \X$ (as in
\cite{abm,b,b2,O}). In the
multigraded case, however, the set $\deg_{\X}(P) =
\{\ua_1,\ldots,\ua_s\} \subseteq \N^r$ may have more than one
element. As we will show below, if $F$ is a separator of $P$ with
$\deg F = \ua_i \in \deg_{\X}(P)$, then the equivalence
class of $F$ in $R/\Ix$, that is, $\overline{F}$, is unique
up to scalar multiplication.

\begin{theorem} \label{boundsonHx}
Let $\X$ be a set of distinct points in $\pnr$, and let $P \in \X$
be any point.  If $\Y = \X \backslash \{P\}$, then there exists a
finite set $S \subseteq \N^r$ such that \[H_{\Y}(\ui) =
\begin{cases} H_{\X}(\ui) & \text{if $\ui \notin D_{S}$} \\
H_{\X}(\ui)-1   & \text{if $\ui \in D_{S}$.}
\end{cases}\]
Equivalently, $\dim_k (I_{\X})_{\ui} \leq \dim_k(I_{\Y})_{\ui}\leq
\dim_k (I_{\X})_{\ui} + 1$ for all $\ui \in \N^r$.\end{theorem}

\begin{proof} The second statement follows from the first since
the formula implies $H_{\X}(\ui)-1 \leq H_{\Y}(\ui) \leq
H_{\X}(\ui)$ for all $\ui \in \N^r$. To prove the first statement,
the short exact sequence \[0 \rightarrow R/(I_{\Y} \cap I_P)
\rightarrow R/I_{\Y} \oplus R/I_P \rightarrow R/(I_{\Y} + I_P)
\rightarrow 0\] implies that \[H_{\Y}(\ui) = H_{\X}(\ui) -
H_{P}(\ui) + H_{R/(I_{\Y}+I_P)}(\ui) ~~ \mbox{for all $\ui \in
\N^r$}\] since $I_{\Y} \cap I_P = I_{\X}$.

Now $R/I_P \cong k[z_1,\ldots,z_r]$, the $\N^r$-graded ring with
$\deg z_i = e_i$. So $H_{P}(\ui) = 1$ for all $\ui \in \N^r$.
Also, \[R/(I_{\Y} + I_P) \cong (R/I_P)/((I_{\Y} + I_P)/(I_P)).\]
So, $(I_{\Y}+I_P)/I_P \cong J$, where $J$ is an $\N^r$-homogeneous
ideal of $k[z_1,\ldots,z_r]$. Thus $H_{R/(I_{\Y}+I_P)}(\ui) = 0$
or $1$ for all $\ui \in \N^r$.

When $H_{R/(I_{\Y} + I_P)}(\ui) = 0$, then $H_{R/(I_{\Y} +
I_P)}(\uj) = 0$ for all $\uj \succeq \ui$. The desired set is then
$S = \min \mathcal{T}$ where $\mathcal{T} = \{\ui \in \N^r ~|~ H_{R/(I_{\Y} +
I_P)}(\ui) = 0\}$. \end{proof}

\begin{corollary}\label{degreeunique}
Suppose $\deg_{\X}(P) = \{\ua_1,\ldots,\ua_s\} \subseteq \N^r$. If
$F$ and $G$ are any two minimal separators of $P$ with $\deg F =
\deg G = \ua_i$, then $G = cF +H$ for some $0 \neq c \in k$ and $H
\in (\Ix)_{\ua_i}$. Equivalently, there exists $0 \neq c \in k$
such that $\overline{G} = \overline{cF} \in R/\Ix$.
\end{corollary}

\begin{proof}
Suppose $F$ and $G$ are separators of $P$ and $\deg F = \deg G =
\ua$ for some $\ua \in \deg_{\X}(P)$.  Suppose that $G \neq cF + H$
for any nonzero scalar $c \in k$ and any $H \in (\Ix)_{\ua}$.  Then the vector space
$(I_{\X},F,G)_{\ua} \subseteq (I_{\Y})_{\ua}$
where $\Y = \X\setminus\{P\}$.  Since $F \not\in
(I_{\X})_{\ua}$, and since $G \not\in (\Ix,F)_{\ua}$, we must have
\[\dim_k (I_{\Y})_{\ua} \geq \dim_k (I_{\X},F,G)_{\ua} \geq \dim_{k} (I_{\X})_{\ua} + 2.\]
However, this inequality contradicts the conclusion of Theorem \ref{boundsonHx}.
\end{proof}

\begin{theorem} \label{teopt} Let $\X$ be a set of distinct points in $\pnr$,
and suppose $F$ is a separator of a point $P \in \X$.
Then $(I_{\X}:F) = I_P$.
\end{theorem}
\begin{proof}
For any $G \in I_P$, $FG \in I_{\X}$ since $FG$ vanishes at all points of  $\X$.  Conversely,
let $G \in (I_{\X}:F)$.  So $GF \in I_{\X} \subseteq I_P$.  Now $F \not\in I_P$, and because
$I_P$ is a prime ideal, we have $G \in I_P$, as desired.
\end{proof}

\begin{corollary}\label{cordim}
 With the hypotheses as in the previous theorem,
\[\dim_k (I_{\X},F)_{\ui} = \dim_k (I_{\X})_{\ui} + 1 ~~\mbox{for all $\ui \succeq \deg F$.}\]
\end{corollary}
\begin{proof}
We have a short exact sequence
\[0 \rightarrow R/(I_{\X}:F)(-\deg F)
\stackrel{\times F}{\longrightarrow} R/I_{\X} \longrightarrow R/(I_{\X},F) \rightarrow 0.\]
By the previous theorem $R/(I_{\X}:F) \cong R/I_P$.  So
\[H_{R/(I_{\X},F)}(\ui) = H_{\X}(\ui) - H_{R/I_P}(\ui - \deg F) ~~\mbox{for all $\ui \in \N^r$.}\]
Now $H_{R/I_P}(\ui) = 1$ for all $\ui \in \N^r$, and equals $0$ otherwise.  The
conclusion follows.
\end{proof}

The main theorem of this section shows that every point
$P \in \X$ has a unique degree
if $\X$ is ACM.

\begin{theorem}\label{uniqueMultsep}
Let $\X$ be any ACM set of points in $\pnr$.   Then for any point $P \in \X$
we have $|\deg_{\X}(P)|  = 1$.
\end{theorem}

\begin{proof}  After a change of coordinates, we can assume that $x_{1,0},\ldots,x_{r,0}$
form a regular sequence on $R/\Ix$, and in particular, for each $i$, $x_{i,0}$
does not vanish at any point of $\X$.  Suppose, for a contradiction, that $P \in \X$ is a point with
$\deg_{\X}(P) = \{\ua_1,\ldots,\ua_t\}$ with
$t = |\deg_{\X}(P)| \geq 2$.
If $\{F_1,\ldots,F_t\}$ are $t \geq 2$ minimal separators of $P$
with $\deg F_i = \ua_i$,
we can reorder and relabel the separators so that $\deg F_i \leq_{lex} \deg F_{i+1}$
for $i = 1,\ldots,t-1$
with respect to the lexicographical order.

For ease of notation, let $F = F_1$ and $G = F_2$ be the two
smallest minimal separators with respect to the lexicographical
order.  Suppose $\deg F = \underline{a} = (a_1,\ldots,a_r)$ and
$\deg G = \underline{b} = (b_1,\ldots,b_r)$.  Set $s = \min\{i
~|~ a_i < b_i\}$;  such an $s$ exists since $\deg F \neq \deg G$
by Corollary \ref{degreeunique}.  Also, let $p = \min\{j ~|~ a_j
> b_j \}$.  Such a $p$ must exist; otherwise $\deg G \succeq \deg
F$, contradicting the fact that $F$ and $G$ are minimal separators
of $P$.

Consider $\underline{c} = (c_1,\ldots,c_r)$ where $c_i = \max\{a_i,b_i\}$.
Since $\underline{c} \succeq \underline{a}$, by Corollary
\ref{cordim}
we must have that $\dim_k (I_{\X},F)_{\underline{c}} = \dim_k (I_{\X})_{\underline{c}} + 1$.
So, a basis for $(I_{\X},F)_{\underline{c}}$ is given by the
$\dim_k (I_{\X})_{\underline{c}}$ basis elements of $(I_{\X})_{\underline{c}}$
and any other form of degree $\underline{c}$ in $(I_{\X},F)_{\underline{c}}
\backslash (I_{\X})_{\underline{c}}$.
One such form is
\[x_{1,0}^{c_1-a_1}\cdots x_{r,0}^{c_r-a_r}F = x_{s,0}^{c_s-a_s}\cdots x_{r,0}^{c_r-a_r}F.\]
Recall,
we are assuming that $x_{i,0}$s form a regular sequence on $R/I_{\X}$, so none
of the $x_{i,0}$'s vanish at any of the points.  As well, $c_i = a_i = b_i$ for $i =1,\ldots,s-1$.

A similar argument implies that $\dim_k (\Ix,G)_{\underline{c}} = \dim_k (\Ix)_{\underline{c}}+1$,
so a basis for $(\Ix,G)_{\underline{c}}$ is given by the
$\dim_k (\Ix)_{\underline{c}}$ basis elements of $(\Ix)_{\underline{c}}$ and
$x_{1,0}^{c_1-b_1}\cdots x_{r,0}^{c_r-b_r}G = x_{p,0}^{c_p-b_p}\cdots x_{r,0}^{c_r-b_r}G$
(because $c_i = b_i$ for $i= 1,\ldots,p-1$).

Since $\underline{c} \succeq \deg G$, and $\underline{c} \succeq \deg F$ we have
\[(I_{\X},F)_{\underline{c}} \subseteq (I_{\Y})_{\underline{c}} ~~\mbox{and}~~
(I_{\X},G)_{\underline{c}} \subseteq
(I_{\Y})_{\underline{c}}.\]
But $\dim_k (I_{\X},F)_{\underline{c}} = \dim_k (I_{\X})_{\underline{c}} + 1$, and since
$\dim_k (I_{\Y})_{\underline{c}} \leq \dim_k (I_{\X})_{\underline{c}}+1$, we must have
$(I_{\X},F)_{\underline{c}} = (I_{\Y})_{\underline{c}}$.  A similar argument implies that
$(I_{\X},G)_{\underline{c}} = (I_{\Y})_{\underline{c}}$.  Hence,
\[(I_{\X},F)_{\underline{c}} = (I_{\X},G)_{\underline{c}}.\]
Because $x_{p,0}^{c_p-b_p}\cdots x_{r,0}^{c_r-b_r}G \in
(I_{\X},G)_{\underline{c}}$,
our discussion about
the basis for $(I_{\X},F)_{\underline{c}}$ implies
\[ x_{p,0}^{c_p-b_p}\cdots x_{r,0}^{c_r-b_r}G = H +
cx_{s,0}^{c_s-a_s}\cdots x_{r,0}^{c_{r}-a_r}F ~~\mbox{with $H \in (I_{\X})_{\underline{c}}$
and $0 \neq c \in k$.}\]
Note that $c \neq 0$ because if $c = 0$, then the right hand side vanishes at all the points
of $\X$, but the left hand side does not.
We thus have
\[
 x_{p,0}^{c_p-b_p}\cdots x_{r,0}^{c_r-b_r}G - cx_{s,0}^{c_s-a_s}\cdots x_{r,0}^{c_{r}-a_r}F
 \in (I_{\X}).\]

Because $x_{1,0},\ldots,x_{r,0}$ form a regular sequence on
$R/I_{\X}$ and since $\X$ is ACM, any permutation of these variables
is again a regular sequence on $R/I_{\X}$.  So, we can assume there is a regular
sequence whose first two elements are $x_{s,0}$ and $x_{p,0}$.  So,
$x_{p,0}^{c_p-b_p}\cdots x_{r,0}^{c_r-b_r}G \in (\Ix,x_{s,0})$, and
since $x_{p,0}$ is regular on $R/(I_{\X},x_{s,0})$ we have
$x_{p+1,0}^{c_{p+1}-{b_{p+1}}}\cdots x_{r,0}^{c_r-b_r}G \in (I_{\X},x_{s,0})$.
Thus
\[x_{p+1,0}^{c_{p+1}-{b_{p+1}}}\cdots x_{r,0}^{c_r-b_r}G = H_1 + H_2x_{s,0} ~~\mbox{with $H_1 \in I_{\X}$
and $H_2 \in R$.}\]
If $Q \in \X \backslash \{P\}$, then $G(Q) = 0$ implies that $H_2(Q) = 0$ since $H_1(Q) = 0$
and $x_{s,0}(Q) \neq 0$.  On the other hand, if we evaluate both sides at $P$ we have
\[
0 \neq x_{p+1,0}^{c_{p+1}-{b_{p+1}}}\cdots x_{r,0}^{c_r-b_r}(P)G(P) =
H_1(P) + x_{s,0}(P)H_2(P) = x_{s,0}(P)H_2(P).\]  But because
$x_{s,0}(P) \neq 0$, this forces $H_2(P) \neq 0$.  So, $H_2$ is a separator of
$P$ with $\deg H_2 = (b_1,\ldots,b_s-1,\ldots,b_p,c_{p+1},\ldots,c_r)$.
Let $F_{\ell}$ be a minimal separator with $\deg F_{\ell} \preceq \deg H_2$.
But then $\deg F_{\ell} \leq_{\operatorname{lex}} \deg G = (b_1,\ldots,b_s,\ldots,b_r)$.  But any
minimal separator whose degree is smaller than $G$ with respect to lex
must have the same degree as $F_1$, i.e.,
$\deg F_1 = \deg F_{\ell}$.  So, $\deg F_1 \preceq \deg H_2$.  But this contradicts the
fact that $a_p > b_p$ and hence $\deg F_1 \not\preceq \deg H_2$.  This
gives the desired contradiction.
\end{proof}

\begin{remark}
If $\X$ is a finite set of ACM points in $\pnr$, then by the above theorem
we know that $|\deg_{\X}(P)| = 1$ for any $P \in \X$.  So we can
talk about {\it the} degree of a point in this situation.
\end{remark}

In the forthcoming paper of Marino \cite{m3}, it is shown that
the converse of the above statement holds in $\popo$, thus
giving a new characterization of ACM sets of points in $\popo$.
We record the precise statement here:

\begin{theorem}\label{marino}
Let $\X$ be a finite set of points in $\popo$.  Then
$\X$ is ACM if and only if
$|\deg_{\X}(P)| = 1$ for all $P \in \X$.
\end{theorem}

However, the converse of Theorem \ref{uniqueMultsep} fails to hold in general
as shown below.

\begin{example} \label{example4}
Let $P_1,\ldots,P_6$ be six points in general position in $\pr^2$.
If $Q_{i,j} = P_i \times P_j \in \pr^2 \times \pr^2$, then let $\X$ be the set
of points
\begin{eqnarray*}
\X & = &\{Q_{1,1},Q_{1,2},Q_{2,1},Q_{2,2},Q_{3,1},Q_{3,2},Q_{4,1},Q_{4,2},Q_{5,2},Q_{5,3}, \\
&&Q_{5,4},Q_{5,5},Q_{5,6},Q_{6,1},Q_{6,3},Q_{6,4},Q_{6,5},Q_{6,6}\}.
\end{eqnarray*}
Using {\tt CoCoA} we found that the resolution of $R/I_{\X}$ is
\[0 \rightarrow R^2 \rightarrow R^{14} \rightarrow R^{33} \rightarrow
R^{34} \rightarrow R^{15} \rightarrow R \rightarrow R/I_{\X} \rightarrow 0 \]
and thus $R/\Ix$ is
not Cohen-Macaulay.  The
bigraded Hilbert function of $\X$ is
\[
H_{\X} = \begin{bmatrix}
1 & 3 & 6 & 6 & \cdots \\
3 & 8 & 14 & 14 & \cdots \\
6 & 14 & 18 & 18 & \cdots \\
6 & 14 & 18 & 18 & \cdots \\
\vdots&\vdots&\vdots&\vdots&\ddots
\end{bmatrix}.\]
If we remove the point $Q_{5,2}$ and compute the
Hilbert function of $\Y = \X \setminus \{Q_{5,2}\}$ we get
\[
H_{\Y} = \begin{bmatrix}
1 & 3 & 6 & 6 & \cdots \\
3 & 8 & 14 & 14 & \cdots \\
6 & 14 & 17 & 17 & \cdots \\
6 & 14 & 17 & 17 & \cdots \\
\vdots&\vdots&\vdots&\vdots&\ddots
\end{bmatrix}.\]
From the Hilbert function, it follows that $\deg_{\X}(Q_{5,2})= \{(2,2)\}$
because
the Hilbert function drops by one for all $\underline{i} \succeq (2,2)$.
By checking all other points in a similar fashion, we have
that $\deg_{\X}(Q_{6,1}) = \{(2,2)\}$ and if we
remove any point $Q_{i,j}$ with $i\leq 4$, then $\deg_{\X}(Q_{i,j}) = \{(2,1)\}$, and if we remove any
point of the form $Q_{i,j}$ with $j \geq 3$, then $Q_{i,j}$ has
only a  minimal separator of degree $(1,2)$.
\end{example}



\begin{thebibliography}{99}

\bibitem{abm} S. Abrescia, L. Bazzotti, L. Marino,
Conductor degree and socle degree.  Matematiche (Catania)  {\bf 56}  (2001) 129--148 (2003).


\bibitem{b} L. Bazzotti, M.  Casanellas,
Separators of points on algebraic surfaces.
J. Pure Appl. Algebra {\bf 207} (2006) 319--326.

\bibitem{b2} L. Bazzotti,
Sets of points and their conductor.
J. Algebra {\bf 283} (2005) 799--820.


\bibitem{CCH} J. Chan, C. Cumming, H.T. H\`a, Cohen-Macaulay multigraded modules.
(2007) Preprint. {\tt arXiv:0705.1839}

\bibitem{Co}
CoCoATeam, CoCoA: a system for doing Computations in Commutative Algebra.
Available at {\tt http://cocoa.dima.unige.it}

\bibitem{gmr} A.V. Geramita, P. Maroscia, L. Roberts,
The Hilbert function of a reduced $k$-algebra.
J. Lond. Math. Soc. (2) {\bf 28} (1983) 443-452.


\bibitem{GuMaRa} S. Giuffrida, R. Maggioni, A. Ragusa,
On the postulation of $0$-dimensional subschemes on a smooth quadric.
Pacific J. Math. {\bf 155} (1992) 251--282.

\bibitem{GuMaRa2} S. Giuffrida, R. Maggioni, A. Ragusa,
Resolutions of $0$-dimensional subschemes of a smooth quadric.
In {\it  Zero-dimensional
schemes (Ravello, 1992)}, de Gruyter, Berlin (1994) 191--204.


\bibitem{GuMaRa3} S. Giuffrida, R. Maggioni, A. Ragusa,
Resolutions of generic points lying on a smooth quadric.
Manuscripta Math. {\bf 91} (1996) 421--444.


\bibitem  {Gu1} E. Guardo,  Schemi di ``Fat Points''.
Ph.D. Thesis, Universit\`a di Messina, 2000.

\bibitem {Gu2} E. Guardo, Fat point schemes on a smooth quadric.
J Pure Appl. Algebra \textbf {162} (2001) 183-208.

\bibitem {Gu3} E. Guardo,  A survey on fat points on a smooth quadric.
{\it Algebraic structures and their representations},  61--87, Contemp. Math.,
{\bf 376}, Amer. Math. Soc., Providence, RI, 2005.


\bibitem{HVT} H. T. H\`a, A. Van Tuyl, The regularity of
points in multi-projective spaces.
J. Pure Appl. Algebra {\bf 187} (2004) 153-167.



\bibitem{m1} L. Marino,
Conductor and separating degrees for sets of points in $\pr^r$ and in
$\popo$.   Boll. Unione Mat. Ital. Sez. B Artic. Ric. Mat.
(8) {\bf 9} (2006) 397--421.

\bibitem{m2} L. Marino,
The minimum degree of a surface that passes through all the points of a
0-dimensional scheme but a point $P$.  {\it Algebraic structures and their representations},
315--332, Contemp. Math., {\bf 376}, Amer. Math. Soc., Providence, RI, 2005.

\bibitem{m3} L. Marino, A characterization of ACM 0-dimensional subschemes
of $\popo$. In preparation.

\bibitem{O}  F. Orecchia,
Points in generic position and conductors of curves with ordinary singularities.
J. London Math. Soc. (2) {\bf 24} (1981) 85--96.


\bibitem{SVT} J. Sidman, A. Van Tuyl,
Multigraded regularity: syzygies and fat points.  Beitr\"age Algebra Geom.
{\bf 47}  (2006) 67--87.


\bibitem{VT} A. Van Tuyl,  Sets of points in multi-projective spaces and their Hilbert function.
Ph.D. Thesis, Queen's University, 2001.

\bibitem{VT1} A. Van Tuyl, The border of the Hilbert function of a set of points in
$\pr\sp {n\sb 1}\times\dots\times\pr\sp {n\sb k}$.  J. Pure Appl. Algebra  {\bf 176}
(2002) 223--247.

\bibitem{VT2} A. Van Tuyl, The Hilbert functions of ACM sets of points in
$\pr\sp {n\sb 1}\times\dots\times\pr\sp {n\sb k}$.  J. Algebra  {\bf 264}
(2003) 420--441.

\bibitem{VT3} A. Van Tuyl, On the defining ideal of a set of points in multi-projective space.
J. London Math. Soc. {\bf 72} (2005) 73-90.
\end{thebibliography}
\end{document}